\documentclass{amsart} \usepackage{amsmath,amssymb,amsfonts}
\usepackage[pdftex]{graphicx}

\usepackage[colorlinks=false]{hyperref}
\hypersetup{colorlinks,citecolor=blue,linktocpage}

\usepackage{sty-va}

\usepackage{todonotes}

\renewcommand\eqref[1]{(\ref{#1})}

\def\kk{\mathrm{k}}

\date{August 25, 2012}

\title{Derived categories of Burniat surfaces and exceptional collections}
 \author{Valery
  Alexeev} \address{Department of Mathematics, University of Georgia,
  Athens, GA 30605, USA} \email{valery@math.uga.edu} \author{Dmitri
  Orlov} \address{Algebraic Geometry Section, Steklov Mathematical
  Institute RAS, Gubkin str. 8, Moscow 119991, RUSSIA}
\email{orlov@mi.ras.ru}

\thanks{The first author was supported by the NSF
  under DMS-1200726. The second author was partially supported by  RFBR grants 10-01-93113, 11-01-00336, 11-01-00568, NSh grant 4713.2010.1,
by AG Laboratory HSE, RF government grant, ag. 11.G34.31.0023.}

\begin{document}

\begin{abstract}
  We construct an exceptional collection $\Upsilon$ of maximal
  possible length~6 on any of the Burniat surfaces with $K_X^2=6$, a
  4-dimensional family of surfaces of general type with $p_g=q=~0$. We
  also calculate the DG algebra of endomorphisms of this collection and
  show that the subcategory generated by this collection is the same
  for all Burniat surfaces.
  The semiorthogonal complement $\cA$ of $\Upsilon$ is an ``almost phantom''
  category: it has trivial Hochschild homology, and $K_0(\cA)=\bZ_2^6$.
\end{abstract}

\maketitle

\tableofcontents

\section{Introduction}
\label{sec:intro}

In a surprising recent paper \cite{DerGodeaux}, B\"ohning, Graf von
Bothmer, and Sosna produced an exceptional sequence of maximal
possible length 11 on the classical Godeaux surface, which is the
$\bZ_5$-quotient of the Fermat quintic surface in $\bP^3$. The
computation is quite involved and is heavily computer-aided. It uses
the $E_8$ root system and a very careful study of effective curves on
the Godeaux surface.

In this paper we make a similar but much easier
computation for Burniat surfaces, which can be described either as
Galois $\bZ_2^2$-covers of $\Bl_3 \bP^2$ or as $\bZ_2^3$-quotients of
$(2,2,2)$-divisors in a product of three elliptic curves.

The Godeaux surface has the same Picard lattice as a del Pezzo surface
of degree~1. The Picard lattice of a Burniat surface is isomorphic to
that of a del Pezzo surface of degree 6. So, essentially the $E_8$
lattice of the classical Godeaux surface
 is replaced by a much smaller lattice
$E_3 = A_2 \times A_1$ with a very small Weil group $S_3\times
S_2$. The Picard number of a Burniat surface is 4, and a maximal
exceptional sequence has length only 6.

Note that the results of \cite{DerGodeaux} apply to a unique surface,
and the results of our computations apply to all Burniat surfaces
which form a 4-dimensional family.  For each of these surfaces $X$, we
find an exceptional collection $\Upsilon=(L_1,\ldots L_6)$ of length 6
consisting of line bundles $L_i$. This collection splits into 3 blocks
of sizes $2+3+1$, and the sheaves in the same block are mutually
orthogonal.

The collection $\Upsilon$ gives a semiorthogonal
decomposition for the bounded derived category of coherent sheaves on
$X$ of the form
$$
\mathbf{D}^b(\operatorname{coh}(X))=\langle L_1, \ldots, L_6, \cA\rangle
$$
The admissible subcategory $\cA$ is ``almost phantom'': it has trivial
Hochschild homology $\mathrm{HH}_*(\cA)=0$ and its Grothendieck group
$K_0(\cA)$ is only the torsion group~$\mathbb{Z}_2^6.$

We also calculate the differential graded (DG) algebra of
endomorphisms of the collection $\Upsilon=(L_1,\ldots L_6)$ and show
that it is formal, i.e.  it is quasi-isomorphic to its cohomology
algebra. This algebra is constant in the family, it is the same
for any Burniat surface.

On the other hand, it is well known (see \cite{BondalOrlov}) that a
smooth variety $X$ with ample canonical class can be uniquely
reconstructed from the derived category
$\mathbf{D}^b(\operatorname{coh}(X))$.  This means that quite
surprisingly, in spite of $\cA$ being ``almost phantom'', all
non-trivial variations of a Burniat surface in the 4-dimensional
family are hidden away in the subcategory $\cA$ and a gluing functor
between this subcategory and a fixed subcategory~$\cD.$

 It would be very interesting to understand if there exist admissible
 subcategories in the bounded derived categories of smooth varieties
 for which not only Hochschild homology but also the Grothendieck
 group $K_0$ is trivial (``phantom''
 categories.)  Arguments for and against existence of ``phantom'' and
 ``almost phantom'' categories were previously discussed in the
 literature (see e.g. \cite{Kuznetsov}) and experts' opinions on this
 vary.
 One candidate for finding a ``phantom'' is Barlow surface. See
 the end of \cite{DiemerKatzarkovKerr} for a related conjecture.


\smallskip Throughout the paper, we work over an algebraically closed
field $\kk$ of characteristic different from 2. The only point where
characteristic is possibly important is the moduli of Burniat
surfaces. For any $\kk$ with $\operatorname{char}\kk\ne 2$, there is a
4-dimensional family coming from line arrangements. Over $\bC$,
\emph{all} Burniat surfaces with $K_X^2=6$ are in this family, see the
discussion on page~\pageref{moduli-discussion}.

\begin{acknowledgments}
  We thank Rita Pardini for providing us with a proof of
  Lemma~\ref{lem:pardini} and for helpful comments.
  We also would like to thank the University of
  Vienna and Ludmil Katzarkov for organizing a workshop on Birational
  geometry and Mirror symmetry during which this project was started.
\end{acknowledgments}

\section{Curves on Burniat surfaces}
\label{sec:lattices}

Burniat surfaces are surfaces of general type with $p_g=q=0$ which
were introduced in \cite{Burniat} (see also \cite{Peters_Burniat}) and
from a different point of view by Inoue \cite{Inoue_Surfaces}.
Burniat surfaces come in several deformation families with $2\le
K_X^2\le 6$. We will consider the unique family with $K_X^2=6$, which
is sometimes called ``primary'' Burniat surfaces. For a detailed study
of Burniat surfaces, including the proof of the fact that they are
Inoue surfaces, see \cite{BauerCatanese_Burniat1}.
Some basic facts
about Burniat-Inoue surfaces can also be found in
\cite[VII.11]{BarthHulekPetersVandeVen2004}.

The easiest way to describe Burniat surfaces with $K_X^2=6$ is as
Galois $\bZ_2^2$-covers of the blowup $\Bl_3\bP^2$ of $\bP^2$ at three
points, a toric del Pezzo surface of degree~6.

Recall from \cite{Pardini_AbelianCovers} that a $\bZ_2^2$-cover
$\pi\colon X\to Y$ with smooth and projective $X,Y$ is determined by
three branch divisors $\bar A,\bar B,\bar C$ and three invertible
sheaves $L_1, L_2, L_3$ on the base $Y$ satisfying fundamental
relations $L_2\otimes L_3\simeq L_1(\bar A)$, $L_3\otimes L_1\simeq
L_2(\bar B)$, $L_1\otimes L_2 \simeq L_3(\bar C)$. These relations
imply that $L_1^2 \simeq \cO_Y(\bar B+\bar C)$, $L_2^2 \simeq
\cO_Y(\bar C+\bar A)$, $L_3^2 \simeq \cO_Y(\bar A+\bar B)$.

One has $X=\Spec_Y \cA$, where the $\cO_Y$-algebra $\cA$ is
$\cO_Y\oplus\oplus_{i=1}^3 L_i\inv$. The multiplication is determined
by three sections in
\begin{displaymath}
  \Hom( L_i\inv \otimes L_j\inv, L_k\inv) =
  H^0( L_i\otimes L_j\otimes L_i\inv),
\end{displaymath}
where $\{i,j,k\}$ is a permutation of $\{1,2,3\}$, i.e. by sections of
the sheaves $\cO_Y(\bar A)$, $\cO_Y(\bar B)$, $\cO_Y(\bar C)$
vanishing on $\bar A$, $\bar B$, $\bar C$.

In our case, the divisors $\bar A=\sum_{i=0}^3 \bar A_i$, $\bar
B=\sum_{i=0}^3 \bar B_i$, $\bar C=\sum_{i=0}^3 \bar C_i$ are the ones
shown in red, blue, and black in the central picture of
Figure~\ref{fig:Burniat} below. The del Pezzo surface has two
different contractions to $\bP^2$ related by a quadratic
transformation, and the images of the divisors form a special line
configuration on either $\bP^2$. We denote by $|h_1|$, resp. $|h_2|$ the linear
systems contracting $\Bl_3\bP^2$ to the left, resp. the right $\bP^2$.

\begin{figure}[h]
  \centering
  \includegraphics[height=1.5in]{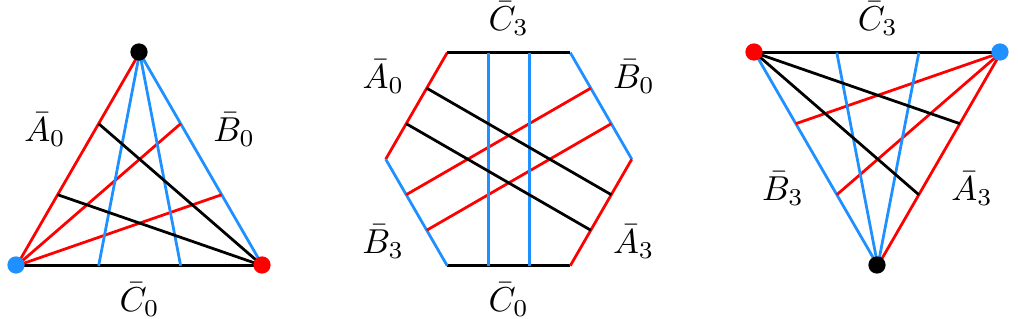}
  \caption{Burniat configuration on $\Bl_3\bP^2$}
  \label{fig:Burniat}
\end{figure}

The curves $\bar A_0$ and $\bar A_3$ are $(-1)$-curves. The curves
$\bar A_1$ and $\bar A_2$ are rational curves with square 0. They are
divisors in a pencil $|f_1|$ which has two reducible fibers $\bar
A_0+\bar C_3$ and $\bar C_0+\bar A_3$. Similarly, $\bar B_1,\bar B_2$
are divisors in a pencil $|f_2|$ and $\bar C_1,\bar C_2$ are divisors
is a pencil $|f_3|$.

We recall the following well known facts about Burniat surfaces:

\smallskip (1) The surface $X$ is smooth iff the configuration of the
branch curves is generic, i.e. the branch divisors do not share
components, at any point no more than two intersect and they belong to
different branch divisors.

\smallskip (2) The map $\pi$ is ramified, with index 2, over the
curves $\bar A_i,\bar B_i,\bar C_i$. Let us denote the corresponding
ramification curves on a Burniat surface by $A_i,B_i,C_i$. One has
$\pi^*(\bar A_i)=2A_i$, etc. For the canonical class, one has
numerically
\begin{displaymath}
  K_X = \pi^*\big( K_Y +\frac12 \sum(\bar A_i+\bar B_i+\bar C_i) \big) =
  \pi^*\big( -\frac12 K_Y\big).
\end{displaymath}
Therefore, $K_X$ is ample and $K_X^2=K_Y^2=6$.

\smallskip (3) One has $h^i(\cO_X)=h^i(\cO_Y)=0$ for $i=1,2$. Thus,
$\chi(O_X)=\chi(O_Y)=~1$. Noether's formula $\chi(\cO) =
(c_1^2+c_2)/12$ implies that $X$ and $Y$ have the same Betti and
Picard numbers.  Hence, $\Pic Y\simeq \bZ^4$ and $\Pic X/\Tors \simeq
\bZ^4$.

\smallskip (4) The torsion subgroup of $\Pic X$ is isomorphic to
$\bZ_2^6$, see \cite{Peters_Burniat}.

\smallskip (5) The fundamental group is an extension of $\bZ_2^3$ by
$\bZ^6$ and is not abelian. This follows from Inoue's construction of
$X$ as a free $\bZ_2^3$-quotient of a divisor in the product of three
elliptic curves, see \cite[p.315]{Inoue_Surfaces},
\cite{BauerCatanese_Burniat1}.

(6) \label{moduli-discussion}
The Burniat configuration of branch curves on $\Bl_3 \bP^2$ is uniquely
determined by a line configuration in $\bP^2$. That one
is described by a 4-dimensional family. Indeed,
for the configuration in, say, the left $\bP^2$, one can fix $\bar
A_0,\bar B_0,\bar C_0$ and $\bar A_1,\bar B_1$; this gives a unique
line configuration with a trivial automorphism group. Then moving the
other 4 lines $\bar A_2,\bar B_2,\bar C_1,\bar C_2$ gives a
4-dimensional family.

Over $\bC$, one knows from Mendez-Pardini \cite{MendezPardini}
that these are \emph{all} the Burniat surfaces
with $K_X^2=6$. They prove that
a deformation of an abelian cover in this case is again an abelian
cover. Since $\bP^2$ does not deform, all the
deformations come from varying the curves, thus varying the lines in $\bP^2$.
The main result of \cite{MendezPardini} is that over
$\bC$ the Burniat surfaces presented above form a connected
component in the moduli space of surfaces of general type.

Moreover, \cite{AlexeevPardini_Burniat} describes the compactification
of this 4-dimensional moduli space obtained by adding stable surfaces,
i.e. surfaces with semi log canonical singularities and ample
canonical class.

\label{lem:lattices}
\begin{lemma}
  The homomorphism $\bar D\mapsto \frac12\pi^*(\bar D)$ defines an
  isomorphism of integral lattices $\frac12\pi^*\colon\Pic Y \to \Pic
  X/\Tors$. One has $\frac12\pi^*(-K_Y) = K_X$.
\end{lemma}
\begin{proof}
  Since $\deg \pi=4$, one has $\frac12 \pi^*(\bar D_1) \cdot \frac12
  \pi^*(\bar D_2) = \frac14 \pi^*(\bar D_1\bar D_2) = \bar D_1\bar
  D_2$. This defines an isomorphism $\Pic Y \otimes \bQ \to \Pic
  X\otimes \bQ$ together with the intersection products. In fact, the
  image of $\Pic Y$ is integral. Indeed, the branch divisors $\bar
  A_i,\bar B_i,\bar C_i$ generate $\Pic Y$ and for each of them
  $\frac12\pi^*(\bar D)$ is an integral cycle. This shows that
  $\frac12\pi^*(\Pic Y) \subset \Pic X/\Tors$ is a sublattice of finite
  index. Since the lattice $\Pic Y$ is unimodular, one must have the
  equality.
\end{proof}

The proof of the lemma shows that numerically we can identify the
curves $\bar D=\bar A_i,\bar B_i,\bar C_i$ downstairs with the
corresponding curves $D:=\frac12\pi^*(\bar D)$ upstairs.  For $i=0,3$,
the curves $A_i,B_i,C_i$ on $X$ are elliptic curves with $D^2=-1$. For
$i=1,2$, they are genus 2 curves with $D^2=0$.

Every point of intersection $\bar P=\bar D_1\cap \bar D_2$ of these
curves on $Y$ gives \emph{only one} point of intersection $P=D_1\cap
D_2$ on $X$ since $D_1D_2=\bar D_1\bar D_2=1$. (Another way to see it:
the divisors $\bar D_1,\bar D_2$ belong to different branch divisors,
so the cover $\pi$ is fully ramified at $P$). Thus, the configuration
of the 12 curves and their intersections on $X$ is exactly the same as
for $Y$, and we can continue to use Figure~\ref{fig:Burniat} to
visualize it.

\begin{notation}\label{notation:ell}
  We give the elliptic curve $A_0$ on $X$ a group structure by fixing
  the point $B_3\cap A_0$ as the origin.  Note that the four points of
  intersection of $A_0$ with the other branch divisors split as $1+3$:
  one point in $A_0\cap B$ and three points in $A_0\cap C$. Our choice
  of the origin is determined by this splitting.

  It is easy to see that the differences between the intersections of
  $A_0$ with $B$ and $C$ are 2-torsion points.  We fix an isomorphism
  of the 2-torsion group $A_0[2]\to\bZ_2^2$ by making the
  identifications
  \begin{displaymath}
    B_3\cap A_0 = 00, \quad C_3\cap A_0 = 10,
    \quad C_1\cap A_0 = 01, \quad C_2\cap A_0 = 11.
  \end{displaymath}
  We make the same identification for the other 5 elliptic curves,
  rotating the hexagon in Figure~\ref{fig:Burniat} cyclically, so 
  that $C_2\cap A_3 = 01$ in $A_3[2]$, and similarly for $B_3,C_3$.

  The subgroup $\bZ.P_{00}+A_0[2] \subset \Pic A_0$ is isomorphic to
  $\bZ\oplus\bZ_2\oplus\bZ_2$. Every element of this subgroup can be
  written as a triple $(a_0^0, a_0^1, a_0^2)$. Similarly, we write
  elements in the groups $\bZ.P_{00}+B_0[2]$, $\bZ.P_{00}+C_0[2]$ as
  triples $(b_0^0, b_0^1, b_0^2)$, $(c_0^0, c_0^1, c_0^2)$.  We use
  similar notation for the three elliptic curves $A_3,B_3,C_3$.

\end{notation}

We thank Rita Pardini for providing us with a proof of the following
Lemma.

\begin{lemma}\label{lem:pardini}
  One has $\cO_{A_0}(K_X) = \cO_{A_0}(-A_0) = \cO_{A_0}( P_{00} )$,
  and similarly for the other 5 elliptic curves.
\end{lemma}
\begin{proof}
  Let $q\colon W\to Y$ be an intermediate double cover corresponding
  to the branch divisor $\bar B+\bar C$, so that $p\colon X\to W$ is
  the double cover for the branch divisor $A'=q\inv(\bar A)$. Note
  that $W$ is singular at the points above $\bar B\cap \bar C$ but is
  smooth in a neighborhood of $\bar A'$, so those points can be
  disregarded in the computation.

  For the double cover $p$ and a connected component $A_0$ of the
  branch divisor, it is immediate that $\cO_{A_0}(A')$ is
  $p^*L|_{A_0}$, where $L$ on $W$ is determined by the equality
  $2L=A'$. One computes that $p^*L=\pi^*L_3-p^* R_2 = \pi^*L_2-p^*
  R_3$, where $R_2$ (resp. $R_3$) is the preimage of $B$ (resp. of
  $C$) on $W$. Plugging this in gives $\cO_{A_0}(A_0) = \cO_{A_0}(A')
  = \cO_{A_0}(-P_{00} )$.
\end{proof}


\begin{theorem}\label{thm:PicX}
  \begin{enumerate}
  \item The homomorphism
    \begin{eqnarray*}
      \phi\colon
      \Pic X &\to& \bZ \times \Pic A_0 \times \Pic B_0 \times \Pic C_0 \\
      L&\mapsto& (d(L)=L\cdot K_X,\  L|_{A_0},\  L|_{B_0},\  L|_{C_0} )
    \end{eqnarray*}
    is injective, and the image is the subgroup of index 3 of
    \begin{displaymath}
      \bZ\times
      (\bZ.P_{00} + A_0[2])
      \times (\bZ.P_{00} + B_0[2])
      \times (\bZ.P_{00} + C_0[2])
      \simeq \bZ^4 \times \bZ_2^6
    \end{displaymath}
    consisting of the elements with $d+a_0^0+b_0^0+c_0^0$ divisible by 3.

  \item $\phi$ induces an isomorphism $\Tors(\Pic X)\to A_0[2]\oplus
    B_0[2]\oplus C_0[2]$.
  \item The curves $A_i,B_i,C_i$, $0\le i\le 3$, generate $\Pic X$.
  \end{enumerate}

\end{theorem}
\begin{proof}
  Fix the contraction $p\colon Y\to \bP^2$ for which $\bar A_0,\bar
  B_0, \bar C_0$ are the exceptional divisors, and let $\bar H$ be the
  generator of $\Pic\bP^2$. Then $-K_Y=3\bar H-\bar A_0-\bar B_0-\bar
  C_0$. Since the lattice $\Pic Y=\bZ\bar H+\bZ\bar A_0+\bZ\bar
  B_0+\bZ\bar C_0$ is unimodular, this implies that the homomorphism
  $\bar\phi\colon\Pic X/\Tors=\Pic Y\to \bZ^4$ induced by $\phi$ is
  injective and the image is the subgroup of index 3 consisting of the
  elements with $d+a_0^0+b_0^0+c_0^0$ divisible by 3.
  This also implies  that $\Tors(\Pic X)=\ker\bar\phi$.

  In Table~\ref{tab:generators}, we write down explicitly the
  restrictions of the curves $A_i,B_i,C_i$ to $\Pic A_i,\Pic B_i,\Pic
  C_i$. From this table, it is obvious that the homomorphism
  $\Tors(\Pic X)\to A_0[2]\oplus B_0[2]\oplus C_0[2]$ is
  surjective. (Indeed, $C_2-C_1$ and $C_2-C_3-B_0$ generate $A_0[2]$;
  similarly for $B_0[2]$, $C_0[2]$.) Since $\Tors(\Pic X)\simeq
  \bZ_2^6$ (see \cite{Peters_Burniat} or
  \cite[p.315]{Inoue_Surfaces}), it must be an isomorphism.

  \begin{table}[htbp!]
    \centering
    \begin{tabular}{|c|c||r|r|r||r|r|r|}
      \hline
      & $d$ & $a_0\ $ & $b_0\ $ & $c_0\ $ & $a_3\ $ & $b_3\ $ & $c_3\ $ \\
      \hline

      $A_0$& 1& $-1$ 00& 0 00& 0 00& 0 00& 1 10& 1 00 \\
      $B_0$& 1& 0 00& $-1$ 00& 0 00& 1 00& 0 00& 1 10 \\
      $C_0$& 1& 0 00& 0 00& $-1$ 00& 1 10& 1 00& 0 00 \\
      \hline
      $A_3$& 1& 0 00& 1 10& 1 00& $-1$ 00& 0 00& 0 00 \\
      $B_3$& 1& 1 00& 0 00& 1 10& 0 00& $-1$ 00& 0 00 \\
      $C_3$& 1& 1 10& 1 00& 0 00& 0 00& 0 00& $-1$ 00 \\
      \hline
      $A_1$& 2& 0 00& 1 01& 0 00& 0 00& 1 11& 0 00 \\
      $A_2$& 2& 0 00& 1 11& 0 00& 0 00& 1 01& 0 00 \\
      \hline
      $B_1$& 2& 0 00& 0 00& 1 01& 0 00& 0 00& 1 11 \\
      $B_2$& 2& 0 00& 0 00& 1 11& 0 00& 0 00& 1 01 \\
      \hline
      $C_1$& 2& 1 01& 0 00& 0 00& 1 11& 0 00& 0 00 \\
      $C_2$& 2& 1 11& 0 00& 0 00& 1 01& 0 00& 0 00 \\
      \hline\hline
      $K_X$& 6& 1 00& 1 00& 1 00& 1 00& 1 00& 1 00\\
      \hline

    \end{tabular}

    \medskip
    \caption{Generators of $\Pic X$ and $K_X$ in symmetric coordinates}
    \label{tab:generators}
  \end{table}

  The image of $A_3$ in $\Pic Y$ is $\bar A_3=\bar H-\bar B_0-\bar
  C_0$. So, together with the curves $A_0,B_0,C_0$ it generates $\Pic
  X/\Tors$, and all the curves together generate $\Pic X$.

\end{proof}

\begin{lemma}
  The coordinates with respect to the triple $(A_3,B_3,C_3)$ are
  related to the coordinates with respect to the triple
  $(A_0,B_0,C_0)$ by the formulas
  \begin{displaymath}
    3a_3^0 = d+a_0^0 -2b_0^0-2c_0^0, \quad
    a_3^1 = a_0^1 + b_0^2 + (d+a_0^0+b_0^0)({\rm mod}\ 2),\quad
    a_3^2 = a_0^2,
  \end{displaymath}
  and similarly for $b_3,c_3$ rotating cyclically.
\end{lemma}
\begin{proof}
  The formula for $a_3^0$ is easy. For $a_3^1,a_3^2$, the formulas
  come from putting Table~\ref{tab:generators} (mod 2), considered as a
  $12\times 19$ matrix with coefficients in $\bZ_2$, in the reduced
  row-echelon form.
\end{proof}

\section{Exceptional collections on $\Bl_3\bP^2$}

It is well-known that the bounded derived category of coherent sheaves
$\mathbf{D}^b(\operatorname{coh}(S))$ on any del Pezzo surface $S$ has
a full exceptional collection (\cite{Orlov_Monoidal}, see also
\cite{KuleshovOrlov}). This is a particular case of a more general
statement about derived categories of blowups.

First, recall the notion of an exceptional collection.
\begin{definition}
  An object $\mathcal{E}$ of a $\kk$-linear triangulated category
  $\mathcal{D}$ is said to be {\sf exceptional} if
  \begin{displaymath}
    \Hom(\mathcal{E}, \mathcal{E}[m])=
    \Ext^m(\mathcal{E}, \mathcal{E})= 0
    \quad \text{for all } m\ne0,
  \end{displaymath}
  and $\Hom(\mathcal{E}, \mathcal{E})=\kk$.
  An ordered set of exceptional objects $\left(\mathcal{E}_1,\ldots
    \mathcal{E}_n\right)$ is called an {\sf exceptional collection} if
  $\Hom(\mathcal{E}_j, \mathcal{E}_i[m])=0$ for $j>i$ and all $m.$
\end{definition}
\begin{definition} An exceptional collection
  $\left(\mathcal{E}_1,\ldots, \mathcal{E}_n\right)$ in a category
  $\mathcal{D}$ is called {\sf full} if it generates the category
  $\mathcal{D},$ i.e. the minimal full triangulated subcategory of
  $\mathcal{D}$ containing all objects $\mathcal{E}_i$ coincides with
  $\mathcal{D}.$ In this case we say that $\mathcal{D}$ has a
  semiorthogonal decomposition of the form
$$
\mathcal{D}=\left\langle \mathcal{E}_1,\ldots,
  \mathcal{E}_n\right\rangle.
$$
\end{definition}
\begin{definition}
  The exceptional collection $\left(\mathcal{E}_1,\ldots
    \mathcal{E}_n\right)$ is said to be {\sf strong} if it satisfies
  the additional condition $\Hom(\mathcal{E}_j, \mathcal{E}_i[m])=0$
  for all $i,j$ and for $m\ne 0.$
\end{definition}

The most studied example of an exceptional collection is the sequence
of invertible sheaves $\langle
\mathcal{O}_{\mathbb{P}^n},\dots,\mathcal{O}_{\mathbb{P}^n}(n)\rangle$
on the projective space $\mathbb{P}^n.$ This exceptional collection is
full and strong.

\begin{definition} \label{def:algexc} The algebra of a strong
  exceptional collection $\Sigma=\left(\mathcal{E}_1,\ldots,
    \mathcal{E}_n\right)$ is the algebra of endomorphisms
  $B_\Sigma=\End(\mathcal{T})$ of the object
  $\mathcal{T}=\mathop\oplus\limits_{i=1}^{n} \mathcal{E}_i.$
\end{definition}

Assume that the triangulated category $\mathcal{D}$ has a full strong
exceptional collection $\Sigma=\left(\mathcal{E}_1,\ldots,
  \mathcal{E}_n\right)$ and $B_{\Sigma}$ is the corresponding
algebra. Denote by $\operatorname{mod}-B_{\Sigma}$ the category of
finite right modules over $B_{\Sigma}.$ There is a theorem according
to which if $\mathcal{D}$ is an {\it enhanced triangulated category}
in the sense of Bondal and Kapranov \cite{BondalKapranov}, then it is
equivalent to the bounded derived category
$\mathbf{D}^b(\operatorname{mod}-B_{\Sigma})$. This equivalence is
given by the functor $\mathbf{R}\Hom(\mathcal{T}, -)$ (see
\cite{BondalKapranov}).

If a full exceptional collection $\Sigma=\left(\mathcal{E}_1,\ldots,
  \mathcal{E}_n\right)$ in an enhanced triangulated category
$\mathcal{D}$ is not strong, then we can consider a DG algebra (or
$A_{\infty}$-algebra) of endomorphisms
$\mathcal{B}_{\Sigma}=\mathbf{R}\Hom(\mathcal{T}, \mathcal{T})$ and
the same functor will induce an equivalence between $\mathcal{D}$ and
the category of perfect objects $\mathrm{Perf}(\mathcal{B}_{\Sigma})$
over $\mathcal{B}_{\Sigma}$.

For this fact and the main results on DG algebras and DG categories,
we refer the reader to \cite{Keller, Keller_congress}.  For the
notions and techniques of $A_{\infty}$-algebras and
$A_{\infty}$-categories, we refer to \cite{Keller_A_infty,
  Seidel_book, Lefevre}.

Any full subcategory $\mathcal{D}\cong
\mathbf{D}^b(\operatorname{coh}(Z))$ of the bounded derived category
of coherent sheaves on a variety $Z$ is enhanced.  Assume that the
subcategory $\mathcal{D}$ is generated by an exceptional collection
$\Sigma=(\mathcal{E}_1,\ldots, \mathcal{E}_n).$ In this case we obtain
an equivalence
$$
\mathbf{R}\Hom(\mathcal{T},
-):\mathcal{D}\stackrel{\sim}{\longrightarrow}
\operatorname{Perf}(\mathcal{B}_{\Sigma}).
$$
If the exceptional collection $\Sigma$ is full then $\mathcal{D}\cong
\mathbf{D}^b(\operatorname{coh}(Z)),$ if the collection is strong then
the DG algebra $\mathcal{B}_{\Sigma}$ is quasi-isomorphic to the
algebra $B_{\Sigma}$ and
$\operatorname{Perf}(\mathcal{B}_{\Sigma})\cong
\mathbf{D}^b(\operatorname{mod}-B_{\Sigma}).$

\begin{theorem}{\rm \cite{Orlov_Monoidal, KuleshovOrlov}}\label{fecBu}
  Let $p: S_K\to \mathbb{P}^2$ be a blowup of the projective plane
  $\mathbb{P}^2$ at a set $K=\{P_1,\dots, P_k\}$ of any $k$ distinct
  points, and let $E_1,\dots, E_k$ be the exceptional curves of the
  blowup. Then the sequence
  \begin{equation}\label{excol}
    \left(\mathcal{O}_{S_K}, p^* \mathcal{O}_{\mathbb{P}^2}(1), p^* \mathcal{O}_{\mathbb{P}^2}(2), \mathcal{O}_{E_1},\dots,\mathcal{O}_{E_k}\right),
  \end{equation}
  where $\mathcal{O}_{E_i}$ are the structure sheaves of the
  exceptional $-1$-curves $E_i$, is a full strong exceptional
  collection on $S_K$.  In particular, there is an equivalence
  \begin{equation}\label{equiv:del}
    \mathbf{D}^b(\operatorname{coh}(S_K))\cong\mathbf{D}^b(\operatorname{mod}- B_K),
  \end{equation}
  where $B_K$ is the algebra of homomorphisms of the exceptional
  collection (\ref{excol}).
\end{theorem}
There are no restrictions on the set of points $K=\{P_1,\dots,P_k\}$
in this theorem and, in particular, we do not need to assume that
$S_K$ is a del Pezzo surface.

Exceptional objects and exceptional collections on del Pezzo surfaces
are well-studied objects.  First, any exceptional object of the
derived category is isomorphic to a sheaf up to translation.  Second,
any exceptional sheaf can be included in a full exceptional
collection.  Third, any full exceptional collection can be obtained
from a given one by a sequence of natural operations on exceptional
collections called {\it mutations}. All these facts can be found in
the paper \cite{KuleshovOrlov}.

An exceptional collection is called a {\sf block} if $\mathcal{E}_i$
and $\mathcal{E}_j$ are mutually orthogonal for all $i\ne j,$
i.e. $\Hom(\mathcal{E}_j, \mathcal{E}_i[m])=0$ for all $m$ and $i\ne
j.$ For example, the sheaves $(\mathcal{O}_{E_1}, \ldots,
\mathcal{O}_{E_k})$ on $S_K$ form a block.

A remarkable fact is that any del Pezzo surface $S_K$ with $3\le
k\le 8$ possesses a full exceptional collection consisting of three
blocks (see \cite{KarpovNogin}).

Consider the del Pezzo surface $Y=\Bl_3\bP^2$ that is a blow up
of projective plane at three points.  As in
Section~\ref{sec:lattices}, we use $f_1,f_2,f_3$ to denote the
divisors defining the special pencils on $Y$, and $h_1,h_2$ for the
divisors defining the contractions $ Y\to \bP^2$.

\begin{theorem}[\cite{KarpovNogin}]
  The following collection on $Y=\Bl_3\bP^2$
  \begin{equation}\label{col:del}
    \Sigma=  \left( \cO_Y, \quad
      \cO_Y(f_1), \cO_Y(f_2), \cO_Y(f_3), \quad
      \cO_Y(h_1), \cO_Y(h_2) \right)
  \end{equation}
  is a full strong exceptional collection that is split into three
  blocks of sizes 1+3+2, and the sheaves in the same block are
  mutually orthogonal.
\end{theorem}

It is easy to see that this collection is strong and to calculate the
algebra of endomorphisms $B_{\Sigma}$ of this exceptional collection.
There are only following nontrivial Hom's between objects of $\Sigma$
$$
\begin{array}{lll}
  \operatorname{Hom}(\cO_Y, \cO_Y(f_i))\cong \kk^2 & \text{for all} & i=1,2,3\\
  \operatorname{Hom}(\cO_Y, \cO_Y(h_j))\cong \kk^3 & \text{for all} &  j=1,2\\
  \operatorname{Hom}(\cO_Y(f_i), \cO_Y(h_j))\cong \kk &
  \text{for all} & i=1,2,3; j=1,2.
\end{array}
$$
With evident composition law they completely define the algebra
$B_{\Sigma}$ of endomorphisms of the collection $\Sigma.$ Thus we obtain

\begin{proposition} There is an equivalence
$$
\mathbf{D}^b(\operatorname{coh}(Y))\cong
\mathbf{D}^b(\operatorname{mod}-B_{\Sigma}),
$$
where $B_{\Sigma}$ is the algebra of endomorphisms of the exceptional
collection $\Sigma$ (\ref{col:del}).
\end{proposition}

\section{Exceptional collections on Burniat surfaces}

We begin with the following numerical statement:

\begin{lemma}\label{lem:lift}
  Let $\bar L_1,\bar L_2$ be two line bundles on $Y$ and $L_1, L_2$ be
  any line bundles on $X$ lifting them under the projection $\Pic X\to
  \Pic X/\Tors=\Pic Y$. Then one has
  \begin{math}
    \chi(X, L_1 \otimes L_2\inv ) = \chi(Y, \bar L_2 \otimes \bar
    L_1\inv)
  \end{math}
\end{lemma}
\begin{proof}
  Denoting $\cO_Y(\bar D)=\bar L_1\otimes \bar L_2\inv$ and
  $\cO_X(D)=L_1\otimes L_2\inv$, by Riemann-Roch:
  \begin{displaymath}
    \chi(X, D) = \frac{D(D-K_X)}2 = \frac{\bar D(\bar D+K_Y)}2 =
    \frac{-\bar D(-\bar D-K_Y)}2 = \chi(Y, -\bar D).
  \end{displaymath}
\end{proof}

\begin{corollary}\label{cor:num_exc}
  For any exceptional sequence $(\bar L_1,\dotsc \bar L_n)$ of line
  bundles on $Y$, $(L_n, \dotsc, L_1)$ is a {\sf numerical}
  exceptional sequence on $X$, i.e. $\chi(L_i\otimes L_j\inv)=0$ for
  $i>~j$.  If in addition $H^0(X, L_i\otimes L_j\inv)=H^0(X,
  L_i\otimes L_j\inv)=0$ for $i>j$ then also $H^1(X, L_i\otimes
  L_j\inv)=0$ and $L_n,\dotsc L_1$ is an exceptional collection on
  $X$.
\end{corollary}

For our example, we start with the exceptional collection on $Y$ of
the previous section and lift it to an exceptional collection on $X$
by checking that for the corresponding differences $\cO(D)=L_i\otimes
L_j\inv$ one has $h^0(D)=0$ and $h^2(D) = h^0(K_X-D)=0$. Note that for
all the differences $D$ involved, both $D$ and $K_X-D$a are of the
form (effective) + (torsion). Their images in $\Pic Y=\Pic X/\Tors$
\emph{are} effective.  However, since the torsion group of $\Pic X$ is
so large, it is possible that $D$ and $K_X-D$ themselves are not
effective for a wise choice of the lifts $L_i$.

\begin{theorem}\label{thm:6-bundles}
  The sequences of line bundles $\langle L_1,L_2,L_3,L_4,L_5,L_6
  \rangle$ and \linebreak $\langle L_1,L_2,L_3,L_4,L_5,L_6' \rangle$
  given in Table~\ref{tab:exc-coll} form exceptional sequences on~$X$
  that split into three blocks of sizes $2+3+1$, and the sheaves in
  the same block are mutually orthogonal.
\end{theorem}
\begin{table}[htbp!]
  \centering
  \begin{tabular}[h]{|c|r||r|r|r||r|r|r|}
    \hline
      & $d$ & $a_0\ $ & $b_0\ $ & $c_0\ $ & $a_3\ $ & $b_3\ $ & $c_3\ $ \\
    \hline
    $L_1$& 3& 0 00& 0 00& 0 00& 1 10& 1 10& 1 10\\
    $L_2$& 3& 1 10& 1 10& 1 10& 0 00& 0 00& 0 00\\
    \hline
    $L_3$& 2& 1 11& 0 01& 0 00& 1 11& 0 01& 0 00\\
    $L_4$& 2& 0 01& 0 00& 1 11& 0 01& 0 00& 1 11\\
    $L_5$& 2& 0 00& 1 11& 0 01& 0 00& 1 11& 0 01\\
    \hline
    $L_6$& 0& 0 00& 0 00& 0 00& 0 00& 0 00& 0 00\\
    \hline
    $L_6'$& 0& 0 10& 0 10& 0 10& 0 10& 0 10& 0 10\\
    \hline
  \end{tabular}
  \medskip
  \caption{Two exceptional collections on a Burniat surface}
  \label{tab:exc-coll}
\end{table}

\begin{remark}
  In terms of the generators, these line bundles can be written as
  follows. We put $L_i=\cO_X(R_i)$ for some divisors $R_i$, and list
  $R_i$.
  $$
  \begin{array}{llll}
    R_1&=&  A_3 +B_0+C_0+A_1-A_2, \\
    R_2&=&  A_0 + B_3+ C_3+A_2-A_1, \\
    R_3&=&  C_2+A_2-C_0-A_3  \ = \   C_1+A_1-C_3-A_0,  \\
    R_4&=&   B_2+C_2-B_0-C_3\ = \   B_1+C_1-B_3-C_0, \\
    R_5&=&  A_2+B_2-A_0-B_3 \ = \  A_1+B_1-A_3-B_0, \\
    R_6&=& 0,\\
    R_6'&=& A_0+B_0+C_0+A_3+B_3+C_3 - K_X \ =\  -R_6'.

  \end{array}
  $$
  One also has
  \begin{displaymath}
    R_1+R_2 \ = \  R_3 + R_4 + R_5 \ = \
    A_0+B_0+C_0+A_3+B_3+C_3
  \end{displaymath}

\end{remark}

\begin{proof}
  We should check that for every difference $D=R_j-R_i$, $i<j$, and
  also for all the differences in the same block one has
  $h^0(D)=h^0(K_X-D)=0$. The task is made easier by the
  $\bZ_3\times\bZ_2=\bZ_6$ symmetry group of this collection and of
  the Burniat configuration.  Almost all the cases are handled by the
  following two elementary considerations:
  \begin{enumerate}
  \item If $DK<0$ or $DK=0$ but $D\ne 0$ then $D$ is not effective.
  \item If for one of the elliptic curves $E=A_0, \dotsc, C_3$ one has
    $DE=0$ and $D|_E\ne 0$ in $E[2]$ then $E$ is in the base locus of
    the linear system $|D|$. Thus, if $D$ is effective then so is
    $D-E$.
  \end{enumerate}
  The remaining cases are handled by Lemma~\ref{lem:not-eff}. We now
  go through the computation.

  \underline{$R_5-R_6$} has $C_0,C_3$ in base
  locus. $(R_5-R_6)-C_0-C_3$ has degree 0 but is not zero itself in
  $\Pic X$. Done.

  Below, we will abbreviate this sequence of arguments as follows:
  ``$\underline{R_5-R_6} \to D-C_0C_3$. $DK=0$ but $D\ne0$.\done'' The
  symbol $D$ will stand for the divisor discussed immediately before,
  e.g. for the above example we first have $D=R_5-R_6$ and then
  $D=R_5-R_6-C_0-C_3$.

  \medskip\noindent \nounderline{$K-(R_5-R_6)$} $\to D-B_0B_3$ $\to
  D-C_0C_3$. $DK=0$ but $D\ne0$. \done

  \smallskip\noindent $R_4-R_6$, $R_3-R_6$ differ from $R_5-R_6$ by a
   $\bZ_3$ symmetry.

  \medskip\noindent \underline{$R_4-R_5$}. $DK=0$ but $D\ne0$.

  \smallskip\noindent \nounderline{$K-(R_4-R_5)$} $\to D-C_0C_3$ $\to
  D-A_0A_3$ $\to D-B_0B_3C_0C_3$. $DK<0$. \done

  \smallskip\noindent \underline{$R_5-R_4$}. $DK=0$ but $D\ne0$.

  \smallskip\noindent \nounderline{$K-(R_5-R_4)$} $\to D-B_0B_3$ $\to
  D-A_0A_3$ $\to D-B_0B_3$. $DK=0$ but $D\ne 0$.\done

  \smallskip\noindent Other $R_i-R_j$, $i,j\in\{3,4,5\}$ are done by
  symmetry.

  \medskip\noindent \underline{$R_2-R_6$} is done by
  Lemma~\ref{lem:not-eff}. \done $R_1-R_6$ differs from it by a
  $\bZ_2$ symmetry.

  \smallskip\noindent \nounderline{$K-(R_2-R_6)$} $\to
  D-A_0B_0C_0$. $DK=0$ but $D\ne0$. \done

  \medskip\noindent \underline{$R_2-R_5$} $\to
  D-B_0C_3$. $DK<0$. \done

  \smallskip\noindent \nounderline{$K-(R_2-R_5)$} $\to D-A_0C_0$ $\to
  D-A_3B_3C_3$. $DK=0$ but $D\ne 0$.\done

  \smallskip\noindent $R_2-R_4$, $R_2-R_3$ differ from $R_2-R_5$ by a
   $\bZ_3$ symmetry.

  \noindent \nounderline{$R_1-R_i$}, $i\in\{4,5,6\}$ differ
  from $R_2-R_i$ by a $\bZ_2$ symmetry.

  \medskip\noindent \underline{$R_1-R_2$}. $DK=0$ but $D\ne0$.

  \medskip\noindent \nounderline{$K-(R_1-R_2)$} $\to D-A_3B_3C_3$ $\to
  D=R_1$. Done by Lemma~\ref{lem:not-eff}.

  \smallskip\noindent $R_2-R_1$ is symmetric to $R_1-R_2$.

  \medskip The differences involving $R_6'$, up to symmetry:

  \noindent
  \underline{$R_5-R_6'$} $\to D-A_0C_0A_3C_3$. $DK<0$. \done

  \smallskip\noindent $K-(R_5-R_6')$ $\to D-B_0B_3$ $\to
  D-A_0C_0A_3C_3$. $DK<0$. \done

  \smallskip\noindent \underline{$R_2-R_6'$} $\to D-A_3B_3C_3$. $DK=0$
  but $D\ne0$. \done

  \smallskip\noindent $K-(R_2-R_6')$ $\to D=R_2$. Done by
  Lemma~\ref{lem:not-eff}.

  This concludes the proof of the theorem.
\end{proof}

\begin{lemma}\label{lem:not-eff}
  The divisor $R_2$ is not effective.
\end{lemma}
\begin{proof}
  Consider the ``corner'' $A_0\cap C_3$ of the hexagon in
  Figure~\ref{fig:Burniat}. We claim that any effective divisor $D\in
  |R_2|$ contains either $A_0$ or $C_3$. Indeed, $R_2 A_0 =1$. So
  either $A_0 \le D$ or $D$ intersects $A_0$ at a unique point giving
  $(1\ 10)$ in $\Pic A_0$ in our coordinates, which is precisely
  $A_0\cap C_3$. But since $R_2 C_3 =0$, $D$ must contain $C_3$.

  The same argument applies to the ``corners'' $B_0\cap A_3$ and
  $C_0\cap B_3$.  Each of the six curves $A_0, \dotsc, C_3$ passes
  through only one of them.  Thus, $D$ must contain at least three of
  the curves $A_0,\dotsc,C_3$. For degree reasons, $D$ must be equal
  to the sum of exactly three of them.  By focusing on the coordinates
  in $A_0[2]$, $B_0[2]$, $C_0[2]$, it follows that one must have
  $D=A_3+B_3+C_3$. But $L_2 A_3 =0$ and $(A_3+B_3+C_3)A_3
  =-1$. Contradiction.
\end{proof}

Consider the exceptional collection $\Upsilon=(L_1, \ldots,
L_6\cong\cO_X)$ and denote by $\cD$ the full triangulated subcategory
of $\mathbf{D}^b(\operatorname{coh}(X))$ generated by this collection.
The subcategory $\cD$ is admissible, i.e. the embedding functor $j:\cD
\to \mathbf{D}^b(\operatorname{coh}(X))$ has right and left adjoint
functors (see \cite{BondalKapranov_Serre_functor}).

Let us calculate the DG algebra of endomorphisms
$\cB_{\Upsilon}=\mathbf{R}\Hom(T, T)$, where $T=\oplus_{i=1}^6 L_i.$
First, we should find all nontrivial Hom's and Ext's between line
bundles $L_i$ and $L_j$ in our collection.

\begin{lemma}\label{lem:hom} For the exceptional collection
  $\Upsilon=(L_1, \ldots, L_6)$ we have
$$
\Hom(L_i, L_j)=0\quad\text{for all}\quad i\ne j.
$$
The same holds for the exceptional collection $\Upsilon'=(L_1, \ldots, L_6')$.
\end{lemma}
\begin{proof}
  Since the collection $\Upsilon$ is exceptional, it remains to check
  only the case when $i<j$ and $L_i, L_j$ are from different blocks.
  But in this case $\Hom(L_i, L_j)\cong H^0(X, \cO(R_j-R_i))=0,$
  because $(R_j-R_i)K<0.$
\end{proof}

\begin{lemma}[\cite{MendezPardini} Lemma 5.6(3)]\label{lem:Pardini}
  If $h^1(X, \cO(\tau))> 0$ for a torsion divisor $\tau\in \Tors(\Pic
  X)$ then the restrictions of $\tau$ on $A_0, B_0,$ and $C_0$ are
  trivial for two of them and is equal to $(10)$ for the third. Hence,
  for all other nonzero torsion divisors $\alpha$ we have $h^1(X,
  \cO(\alpha))=0$ and $h^2(X,
  \cO(\alpha))=\chi(\cO(\alpha))=\chi(\cO)=1.$
\end{lemma}

\begin{lemma}\label{lem:ext} For the exceptional collection
  $\Upsilon=(L_1, \ldots, L_6)$ we have
$$
\Ext^1(L_i, L_j)=0\quad\text{for all}\quad i, j.
$$
The same results hold after replacing $L_6$ by $L_6'$.
\end{lemma}
\begin{proof}
  Since the collection $\Upsilon$ is exceptional, it remains to check
  only the cases when $i<j$ and $L_i, L_j$ are from different blocks.

  Consider two line bundles $L_1$ and $L_3$ from the first and the
  second blocks. The line bundle $L_3\otimes L_1^{-1}$ is isomorphic
  to $\cO_{X}(-A_0+\alpha),$ where $\alpha$ is a torsion divisor.  It
  follows from Lemma \ref{lem:Pardini} and Table \ref{tab:exc-coll}
  that $h^1(X, \cO(\alpha))=0$ and the restriction $\alpha|_{A_0}$ is
  not trivial.  Therefore $h^0(A_0, \alpha|_{A_0})=0$ and from the
  short exact sequence
  \begin{displaymath}
    0\longrightarrow \cO(-A_0+\alpha)\longrightarrow
    \cO(\alpha)\longrightarrow \cO_{A_0}(\alpha)\longrightarrow 0
  \end{displaymath}
  we obtain that $H^1(X, \cO(-A_0+\alpha))=0.$ Similar considerations
  work for all line bundles from the first and the second blocks. We
  only have to replace $A_0$ by another elliptic curve $B_0, C_0, A_3,
  B_3, C_3,$ respectively.

  Now consider $L_3$ and $L_6=\cO_X.$ The line bundle $L_3^{-1}$ is
  isomorphic to $\cO_X(-B_0-C_3+\beta),$ where $\beta$ is a torsion
  divisor.  Consider the following short exact sequences
  \begin{displaymath}
    \begin{array}{l}
      0\longrightarrow \cO(-B_0+\beta)\longrightarrow
      \cO(\beta)\longrightarrow \cO_{B_0}(\beta)\longrightarrow 0\\
      0\longrightarrow \cO(-B_0-C_3+\beta)\longrightarrow
      \cO(-B_0+\beta)\longrightarrow \cO_{C_3}(-B_0+\beta)\longrightarrow 0,
    \end{array}
  \end{displaymath}
  It follows from Lemma \ref{lem:Pardini} and Table \ref{tab:exc-coll}
  that $h^1(X, \cO(\beta))=0$ and the restrictions $\beta|_{B_0}$ is
  not trivial.  Therefore $h^1(X, \cO(-B_0+\beta))=0.$ From the second
  exact sequence, noting that $\deg \cO_{C_3}(-B_0+\beta)<0$, we
  obtain $h^1(X, \cO(-B_0-C_3+\beta))=0$.  Similar considerations work
  for any line bundle $L_i, i=3,4,5$ from the second block and
  $L_6=\cO_X.$

  Finally, let us take $L_1$ and $L_6=\cO_X.$ The line bundle
  $L_1^{-1}$ is isomorphic to $\cO_X(-A_0-C_1+\gamma),$ where $\gamma$
  is a torsion divisor.  Consider the following short exact sequences
  \begin{displaymath}
    \begin{array}{l}
      0\longrightarrow \cO(-A_0+\gamma)\longrightarrow
      \cO(\gamma)\longrightarrow \cO_{A_0}(\gamma)\longrightarrow
      0\\
      0\longrightarrow \cO(-A_0-C_1+\gamma)\longrightarrow
      \cO(-A_0+\gamma)\longrightarrow \cO_{C_1}(-A_0+\gamma)\longrightarrow 0
    \end{array}
  \end{displaymath}
  It follows from Lemma \ref{lem:Pardini} and Table \ref{tab:exc-coll}
  that $h^1(X, \cO(\gamma))=0$ and the restriction $\gamma|_{A_0}$ is
  not trivial.  Therefore $h^1(X, \cO(-A_0+\gamma))=0.$ From the
  second exact sequence, noting that $\deg \cO_{C_1}(-A_0+\gamma)<0$,
  we obtain that $h^1(X, \cO(-A_0-C_1+\gamma))=~0$.

  The case of the line bundles $L_2$ and $L_6$ works by symmetry. The
  $\Ext$ groups involving $L_6'$ are handled the same way.
\end{proof}

Lemmas \ref{lem:hom} and \ref{lem:ext} and calculation of Euler
characteristic immediately imply the following proposition.

\begin{proposition}\label{prop:ext2} For the exceptional collection
  $\Upsilon=(L_1, \ldots, L_6)$ the nontrivial Ext groups are the
  following
$$
\begin{array}{lll}
  1)& \Ext^2(L_i, L_j)\cong \kk & \text{for}\quad i=1,2;\; j=3,4,5\\
  2)& \Ext^2(L_j, L_6)\cong \kk^2 & \text{for}\quad j=3,4,5\\
  3)& \Ext^2(L_i, L_6)\cong \kk^3 & \text{for}\quad i=1,2.
\end{array}
$$
The same results hold after replacing $L_6$ by $L_6'$.
\end{proposition}

Let us consider the DG algebra $\cB_{\Upsilon}=\mathbf{R}\Hom(T, T),$
where $T=\oplus_{i=1}^6 L_i.$ The minimal model $H^*(\cB_{\Upsilon})$
of this DG algebra is an $A_{\infty}$-algebra that by Proposition
\ref{prop:ext2} has only two nontrivial graded terms
$$
H^0(\cB_{\Upsilon})\cong \bigoplus_{i,j}\Hom(L_i, L_j)\cong\kk^6,\quad \text{ and}\quad
H^2(\cB_{\Upsilon})=\bigoplus_{i,j}\Ext^2(L_i, L_j).
$$
This $A_\infty$-algebra is actually an $A_\infty$-category on our 6
objects $L_1,...,L_6.$ The $H^0(\cB_{\Upsilon})$ is the sum of
identity morphisms of these 6 objects.  By standard theorems (see
\cite{Seidel_book} Lemma 2.1 or \cite{Lefevre} Th 3.2.1.1), any such
$A_\infty$-category is equivalent to a strict $A_\infty$-category.
This means we can assume that $m_l(...,
\operatorname{id}_{L_i},...)=0$ for all objects $L_i$ and all $l>2.$

\label{no-higher-ms}
Thus all nontrivial $m_l$ for $l>2$ can exist only if all elements are from  $H^2(\cB_{\Upsilon}).$
But they are also
trivial for dimension reasons: the degree of $m_l$ is $2-l$, so
$m_l(a_1,\ldots, a_k)\in H^{l+2}(\cB_{\Upsilon})=0.$
Thus, one has
$m_l=0$ for all $l>2.$
This means that the DG algebra $\cB_{\Upsilon}$ and the $A_{\infty}$-algebra
$H^*(\cB_{\Upsilon})$ are formal and are quasi-isomorphic to a usual graded
algebra. Moreover, all compositions of all elements of degree 2 are
0. Thus, the algebra $H^*(\cB_{\Upsilon})$ is not changed under a
deformation of a Burniat surface $X.$ Thus, we obtain the following:

\begin{proposition} Let $\Upsilon=(L_1, \ldots, L_6)$ be the
  exceptional collection constructed above. Then the DG algebra
  $\cB_{\Upsilon}=\mathbf{R}\Hom(T, T),$ where $T=\oplus_{i=1}^6 L_i,$
of
  endomorphism of this collection
is formal, i.e. it is quasi-isomorphic
  to its cohomology algebra $H^*(\cB_{\Upsilon}).$  The admissible subcategory
  $\cD\subset\mathbf{D}^b(\operatorname{coh}(X))$ generated by $\Upsilon=(L_1, \ldots, L_6)$ is the same for
  all Burniat surfaces $X,$ and $ \cD\cong
  \mathrm{Perf}(H^*(\cB_{\Upsilon})).$
\end{proposition}
\begin{remark}
  Note that a Burniat surface $X$ can be reconstructed from the
  bounded derived category of coherent sheaves
  $\mathbf{D}^b(\operatorname{coh}(X)),$ because the canonical class
  $K_X$ is ample (see \cite{BondalOrlov}).
\end{remark}

As it was mentioned above, the subcategory $\cD$ generated by
an exceptional collection is admissible, i.e. the embedding functor
has right and left adjoint functors. Thus we obtain a semiorthogonal
decomposition
$$
\mathbf{D}^b(\operatorname{coh}(X))=\langle \cD, \cA \rangle
$$
where $\cA$ is a left orthogonal to $\cD,$ i.e. $\cA$ consists of all
object $A$ such that $\Hom(A, D)=0$ for all objects $D\in \cD.$

A semiorthogonal decomposition implies direct sum decompositions for
K-theory and Hochschild homology (see, for example, \cite{Kuznetsov} for Hochschild homology).
Thus, we obtain
$$
K_0(X)=K_0(\cD)\oplus K_0(\cA), \quad\text{and} \quad
\mathrm{HH}_*(X)=\mathrm{HH}_*(\cD)\oplus\mathrm{HH}_*(\cA)
$$
All these invariants are defined for $\cD$ and $\cA,$ because they
have induced DG enhancements.

There is an isomorphism for Hochschild homology of a smooth projective
variety
$$
\mathrm{HH}_i(X)=\bigoplus_p H^{p+i}(X, \Omega^p X)
$$
For a Burniat surface it gives that $\mathrm{HH}_i(X)=0$ when $i\ne
0,$ and $\mathrm{HH}_0(X)\cong \kk^6.$ The subcategory $\cD$ as a
subcategory generated by an exceptional collection of length 6 has the
same Hochschild homology.  Hence, we obtain that
$\mathrm{HH}_*(\cA)=0.$

Recall that Bloch conjecture for Burniat surfaces was proved by Inose
and Mizukami in \cite{InoseMizukami}, i.e. the Chow group $\mathrm{CH}^2(X)$ of 0-cycles is isomorphic to~$\mathbb{Z}.$  Therefore, for $K$-theory
we obtain that $K_0(X)\cong \mathbb{Z}^6\oplus \mathbb{Z}_2^6.$ It is
evident that $K_0(\cD)\cong \mathbb{Z}^6.$   Hence,
$K_0(\cA)\cong \mathbb{Z}_2^6.$

Let us summarize what we have just established.
\begin{theorem} For any Burniat surface $X$ we have a semiorthogonal
  decomposition
$$
\mathbf{D}^b(\operatorname{coh}(X))=\langle \cD, \cA \rangle,
$$
where $\cD$ is a subcategory generated by an exceptional collection
$\Upsilon=(L_1, \ldots, L_6).$ The category $\cD$ is the same for all
Burniat surfaces. The category $\cA$ has trivial Hochschild homology
and $K_0(\cA)=\mathbb{Z}_2^6.$ The same holds after replacing $L_6$
by~$L_6'$.
\end{theorem}

\section{Further remarks}

It is proved in \cite{Orlov_gendim} that for any quasi-projective
scheme $Z$ of dimension $d$ and a very ample line bundle $L$, the
object $\cE=\oplus_{i=0}^{d} L^{i}$ is a classical generator for the
triangulated category of perfect complexes $\mathrm{Perf}(Z).$ This
means that $\cE$ idempotently generates the category
$\mathrm{Perf}(Z),$ i.e. the minimal idempotent complete triangulated
subcategory containing it coincides with $\mathrm{Perf}(Z)$ (see
\cite{Orlov_gendim} for details).

This result can be easily made a little stronger: it is sufficient to
assume that $L=f^*L'$ is a pull back of a very ample line bundle $L'$
under a finite morphism $f: Z\to W$ (i.e. that $L$ is an ample and semiample line
bundle). Indeed, since $L'$ is very ample on $W$, the object
$\cE'= \oplus_{i=0}^{d} (L')^{i}$ idempotently generates the category
$\Perf(W)$. Hence, all $L^m=f^*(L')^m$ belong to the idempotent
complete subcategory generated by $\cE=f^*\cE'$. Some power of $L$ is very
ample, so the minimal derived category containing $\cE=\oplus_{i=0}^{d} L^{i}$ is the whole $\Perf(Z)$.

On a Burniat surface, $\cO_X(2K_X)=f^*\cO_Y(-K_Y)$ is ample and
semiample. Thus, $\cE=\cO_X\oplus\cO_X(2K_X)\oplus\cO_X(4K_X)$
is a classical generator for $\Perf(X)=\mathbf{D}^b(\operatorname{coh}(X)).$


By Keller's results \cite{Keller}, this means that the bounded derived
category \linebreak $\mathbf{D}^b(\operatorname{coh}(X))$ is equivalent to the
triangulated category of perfect objects $\mathrm{Perf}(\cB_{\cE})$
over the DG algebra of endomorphisms $\cB_{\cE}=\mathbf{R}\Hom(\cE,
\cE).$ Furthermore, we can try to produce a classical generator for
the full admissible subcategory $\cA$ considering a projection of
$\cE$ to $\cA.$ Since $\cO_X$ belongs to $\cD$, its projection is
trivial and we have to take in account only projections of
$\cO_X(2K_X)$ and $\cO_X(4K_X)$ to $\cA.$ It will be very
interesting to consider this generator for the subcategory $\cA$ and
to calculate the DG (or $A_{\infty}$) algebra of endomorphisms for it.

\bigskip

One can speculate whether our results are but one example of
a general phenomenon. 

\begin{question}
  Is it true that for any exceptional collection $\Upsilon$ of maximal
  length on a smooth surface $X$ with ample $K_X$ and with $p_g=q=0$,
  the DG algebra of endomorphisms of $\Upsilon$ does not change under
  small deformations of the complex structure on $X$?
\end{question}

Certainly, if the surface $X$ is a ``fake del Pezzo surface'' and the
exceptional collection $\Upsilon=(L_n,\dotsc, L_1)$ lifts a nice enough
exceptional collection $(\bar L_1,\dotsc, \bar L_n)$ on the corresponding
del Pezzo surface $Y$, then as in Lemma \ref{lem:hom} it follows that
$\Hom(L_i,L_j)=0$. If one is lucky and $\Ext^1(L_i,L_j)=0$ as well then,
as we explained on page~\pageref{no-higher-ms}, the DG
algebra of endomorphisms of $\Upsilon$ is formal and, moreover, it
does not have any deformations.

But even if we are not in such a ``lucky'' situation, the general
take-away from our computations seems to be that under the
correspondence between a true del Pezzo surface $Y$ and its ``fake
partner'' $X$ that switches $-K_Y$ and $K_X$, the groups $\Ext^0(\bar
L_j,\bar L_i)$, $i\ne j$, switch their places with
$\Ext^2(L_i,L_j)$. On a del Pezzo, the $\Hom$ groups for \emph{some}
exceptional collection completely describe
$\mathbf{D}^b(\operatorname{coh}(Y))$ and thus also the position of
$Y$ in the moduli space. But on the surface of general type, all this
information is pushed to the $\Ext^2$'s which for degree reasons lead
to a rigid DG algebra and say \emph{nothing} about the position of $X$
in the moduli space.

\bibliographystyle{amsalpha} \renewcommand{\MR}[1]{}
\providecommand{\bysame}{\leavevmode\hbox to3em{\hrulefill}\thinspace}
\providecommand{\MR}{\relax\ifhmode\unskip\space\fi MR }
\providecommand{\MRhref}[2]{%
  \href{http://www.ams.org/mathscinet-getitem?mr=#1}{#2}
}
\providecommand{\href}[2]{#2}

\end{document}